\documentclass[leqno,12pt]{amsart}
\usepackage{amsmath,amstext,amssymb,amsopn,amsthm,mathrsfs}

\numberwithin{equation}{section}

\allowdisplaybreaks
 
\DeclareMathOperator{\domain}{Dom}
\DeclareMathOperator{\spann}{span}
\DeclareMathOperator{\sgn}{sgn}

\newtheorem{theor}{Theorem}[section]
\newtheorem{propo}[theor]{Proposition}
\newtheorem{lemma}[theor]{Lemma}
\newtheorem{corollary}[theor]{Corollary}

\newtheorem{defn}{Definition}[section]

\newtheorem*{rem*}{Remark}

\newcommand{\el}{l}

\begin{document}
\footnotetext{
\emph{2010 Mathematics Subject Classification:}  42A50, 42A38.\\
\emph{Key words and phrases:} orthogonal expansions, conjugacy, Riesz transforms, 
	conjugate Poisson integrals
	
Research of both authors supported by MNiSW Grant N201 417839.
}

\title[A symmetrized conjugacy scheme]{A symmetrized conjugacy scheme for orthogonal expansions}


\author[A. Nowak]{Adam Nowak}
\address{Adam Nowak, \newline
			Instytut Matematyczny,
      Polska Akademia Nauk, \newline
      \'Sniadeckich 8,
      00--956 Warszawa, Poland \newline
			\indent and \newline
			Instytut Matematyki i Informatyki,
      Politechnika Wroc\l{}awska,       \newline
      Wyb{.} Wyspia\'nskiego 27,
      50--370 Wroc\l{}aw, Poland      
      }
\email{Adam.Nowak@pwr.wroc.pl}

\author[K. Stempak]{Krzysztof Stempak}
\address{Krzysztof Stempak,     \newline
      Instytut Matematyki i Informatyki,
      Politechnika Wroc\l{}awska,       \newline
      Wyb{.} Wyspia\'nskiego 27,
      50--370 Wroc\l{}aw, Poland      }  
\email{Krzysztof.Stempak@pwr.wroc.pl}

\begin{abstract}
We establish a symmetrization procedure in a context of general orthogonal expansions associated with
a second order differential operator $L$, a `Laplacian'. Combined with a unified conjugacy scheme
furnished in our earlier article it allows, via a suitable embedding, to associate
a differential-difference `Laplacian' $\mathbb{L}$ with the initially given orthogonal system 
of eigenfunctions of $L$, so that
the resulting extended conjugacy scheme has the natural classical shape. This means, in particular,
that the related `partial derivatives' decomposing $\mathbb{L}$ are skew-symmetric 
in an appropriate $L^2$ space and they commute with Riesz transforms and
conjugate Poisson integrals. The results shed also some new light on the question of 
defining higher order Riesz transforms for general orthogonal expansions.
\end{abstract}

\maketitle

\section{Introduction} \label{sec:intro}
The seminal article of Muckenhoupt and E. M. Stein \cite{MS} initiated the investigation of conjugacy
for discrete and continuous nontrigonometric orthogonal expansions. 
In the recent years a considerable activity could be observed in studying conjugacy, or better Riesz
transforms, for orthogonal expansions  in one-dimensional and multi-dimensional settings related to 
general second order differential operators. 

A variety of papers has been devoted to the study of objects 
being ingredients of conjugacy notions defined by different authors in many particular situations.
In connection to a dynamic development of investigation of conjugacy problem in different settings, 
a natural demand appeared on a general and universal definition of Riesz transforms. 
The authors' paper  \cite{NS1} was an attempt to provide a reasonable answer to this important 
demand by offering a unified conjugacy scheme that includes definitions of Riesz transforms and 
conjugate Poisson integrals for a broad class of expansions.
The postulated definitions were supported by a ``good'' $L^2$-theory, existence of
Cauchy-Riemann type equations, and numerous examples existing in the literature 
which are covered by the scheme.

There is, however, a shortcoming of this unified conjugacy scheme manifested in a lack of symmetry 
in the decomposition
$$
L=A+\sum_{j=1}^d \delta^*_j\delta_j,
$$
of a given second order partial differential operator $L$, a `Laplacian', acting on functions on 
a $d$-dimensional domain
$\mathcal X=(b,c)^d$, $\infty\le b<c\le\infty$. Here $A\ge0$ is a constant, $\delta_j$ are `partial
derivatives' associated to $L$ (first order partial differential operators, $\delta_j$ acts on the $j$th
coordinate),  and $\delta_j^*$ are their formal adjoints in an appropriate $L^2$ sense. Riesz transforms
of first order defined in \cite{NS1} are formally given by $R_j=\delta_jL^{-1/2}$ (or by
$R_j=\delta_jL^{-1/2}\Pi_0$, see \cite{NS1} for details), but a replacement of $\delta_j$ by
$\delta_j^*$ in this definition is, in general, inappropriate since it may result in an operator
taking $L^2$ functions out of $L^2$.

Asymmetry of the decomposition of $L$  has, in fact, a deep impact onto the whole conjugacy scheme
postulated in \cite{NS1}. To be precise, taking into account existing examples, it seems that the case of
an operator $L$ acting on the whole space $\mathbb{R}^d$ is not really affected by this asymmetry.
Therefore, in what follows we consider only the case  $(b,c)\neq\mathbb R$, and assume, without any
loss of generality, that $0=b<c\le\infty$. Then a possible way of overcoming the lack of symmetry
is provided by a symmetrization procedure, which is the purpose and the main achievement of the paper.
This procedure is to some extent inspired by a situation of certain orthogonal systems appearing in the
theory of Dunkl operators, see Section \ref{sec:sym} for more comments.
`Partial derivatives' emerging from the symmetrization procedure, contrary to $\delta_j$'s, 
are skew-symmetric as it happens in many classical cases including, in particular, 
the usual Euclidean partial derivatives, Dunkl operators, left-invariant vector fields on Lie groups, etc.

Throughout the paper we use a fairly standard notation. 
The symbols $\Delta$ and $dx$ will always refer to the Euclidean Laplacian,
$\Delta=\sum_{i=1}^d\partial_{x_i}^2$, and Lebesgue measure acting, or considered,
on an appropriate domain in $\mathbb{R}^d$ like, for instance, $\mathbb{R}^d_{+}=(0,\infty)^d$.
The symbol $\mathbb{N}$ is used to denote the set of nonnegative integers, 
$\mathbb{N}=\{0,1,2,\ldots\}$. Finally, $\langle \cdot,\cdot \rangle_{\mu}$ denotes the canonical
inner product in an appropriate $L^2$ space, where $\mu$ is a given measure.

\section{Initial situation} \label{sec:ini}

Our starting point is the situation discussed in \cite[Section 2]{NS1}, where the concept of studying
conjugacy for orthogonal expansions is based on the existence of a second order differential operator
playing a similar role to that of the standard Laplacian in the classical harmonic analysis.
Below, $d \ge 1$ will always denote the dimension. 

We first consider the following one-dimensional
objects, which in a while will serve as building blocks of $d$-dimensional product structure:
\begin{itemize}
\item an open (possibly unbounded) interval $X \subset \mathbb{R}$;
\item a system $\{\mu_i : i=1,\ldots,d\}$ of absolutely continuous measures on $X$, 
	$\mu_i(dx_i)=w_i(x_i)dx_i$ with strictly positive densities $w_i \in C^2(X)$;
\item a system $\{L_i : i=1,\ldots,d\}$ of second order differential operators defined on $C_c^2(X)$.
\end{itemize}
Here, in this paper, we exclude the case $X=\mathbb{R}$ (see the comment in Section \ref{sec:intro}),
so without any loss of generality we may assume that $X=(0,c)$, $0 < c \le \infty$. 
For each of the operators $L_i$,
in order to ensure existence of the associated `derivative',  we assume the decomposition
\begin{equation} \label{decLi}
L_i = a_i + \delta_i^* \delta_i,
\end{equation}
where $a_i$ is a nonnegative constant, and $\delta_i$ is a first order differential operator
(a `derivative') of the form
$$
\delta_i = p_i(x_i) \frac{\partial}{\partial x_i} + q_i(x_i)
$$
with real coefficients $p_i \in C^2(X)$, $q_i \in C^1(X)$, $p_i(x_i) \neq 0$ for $x_i \in X$;
here $\delta_i^*$ represents the formal adjoint of $\delta_i$ in $L^2(X,\mu_i)$,
$$
\delta_i^* = -p_i(x_i)\frac{\partial}{\partial x_i} + q_i(x_i) - p_i(x_i) \frac{w'_i(x_i)}{w_i(x_i)}
	- p'_i(x_i)
$$
determined by the identity
$$
\langle \delta_i \varphi, \psi \rangle_{\mu_i} 
	= \langle \varphi, \delta_i^* \psi \rangle_{\mu_i}, \qquad \varphi,\psi \in C_c^1(X).
$$
Notice that $\delta^*\neq - \delta$.
Thus, a posteriori, each $L_i$ is a linear operator with continuous real-valued coefficients and
negative leading term coefficient,
\begin{align*}
L_i & = - p_i^2(x_i)\frac{\partial^2}{\partial x_i^2} - \bigg[ 2p_i(x_i)p'_i(x_i) + p_i^2(x_i)
	\frac{w'_i(x_i)}{w_i(x_i)} \bigg] \frac{\partial}{\partial x_i}\\ 
	& \quad +  q_i^2(x_i)
		- \big( p_i(x_i)q_i(x_i)\big)' - p_i(x_i)q_i(x_i) \frac{w'_i(x_i)}{w_i(x_i)} + a_i.
\end{align*}
Moreover, \eqref{decLi} implies that each $L_i$ is symmetric and nonnegative
on $C_c^2(X)\subset L^2(X,\mu_i)$.

Now we are in a position to specify a $d$-dimensional setting that is suitable for further development.
We equip the space $\mathcal{X} = X \times \ldots \times X$ ($d$ times) with the product measure
$$
\mu = \mu_1 \otimes \ldots \otimes \mu_d.
$$
We consider the $d$-dimensional `Laplacian' $L$ (more precisely, $L$ is a generalization of $-\Delta$)
defined initially on $C_c^2(\mathcal{X})$ by
$$
L = L_1 + \ldots + L_d,
$$
where each $L_i$ is understood as a one-dimensional operator acting on the $i$th axis.
Note that in view of the previous assumptions, $L$ admits the decomposition
$$
L = A + \sum_{i=1}^d \delta_i^* \delta_i, \qquad \textrm{where} \quad A = \sum_{i=1}^d a_i \ge 0;
$$
here the indices of $\delta$ and $\delta^*$ indicate also on which axes actions of these operators
take place. 

Next, we introduce an orthogonal system associated with $L$. With no loss of generality we may
restrict to $L^2$-normalized systems. Assume that for each $i=1,\ldots,d$, there exists an orthonormal
and complete in $L^2(X,\mu_i)$ system $\{\varphi^{(i)}_{k_i} : k_i \in \mathbb{N}\}$ 
consisting of eigenfunctions
of $L_i$, with the corresponding eigenvalues $\{\lambda_{k_i}^{(i)} : k_i \in \mathbb{N}\}$, i.e.
$L_i \varphi^{(i)}_{k_i} = \lambda_{k_i}^{(i)} \varphi^{(i)}_{k_i}$. Here and below we assume for
simplicity that $\varphi^{(i)}_{k_i} \in C^{\infty}(X)$, but in fact much less regularity is needed
(we omit a discussion in this direction since it could affect the main line of thought of the paper).
For a multi-index $k = (k_1,\ldots,k_d) \in \mathbb{N}^d$ we define
$$
\varphi_k = \varphi_{k_1}^{(1)} \otimes \ldots \otimes \varphi_{k_d}^{(d)}.
$$
Then $\{\varphi_k : k \in \mathbb{N}^d\}$ is an orthonormal basis in $L^2(\mathcal{X},\mu)$ consisting
of eigenfunctions of $L$,
$$
L \varphi_k = \lambda_k \varphi_k, \qquad \textrm{where} \quad 
	\lambda_k = \lambda_{k_1}^{(1)} + \ldots + \lambda_{k_d}^{(d)}.
$$
In addition, $\varphi_k \in C^{\infty}(\mathcal{X})$.

We impose the following technical assumptions on the systems of eigenvalues and eigenfunctions,
which seem to be unavoidable on the considered level of generality. For every $i=1,\ldots,d$, we assume
that the one-dimensional eigenvalues are indexed in the (strictly) ascending order,
$\lambda_0^{(i)}<\lambda_1^{(i)}< \lambda_2^{(i)}< \ldots$, and $\lim_{k_i} \lambda_{k_i}^{(i)}=\infty$.
Consequently, the set $\{\lambda_k : k \in \mathbb{N}^d\}$ of multi-dimensional eigenvalues may be arranged
into an increasing and divergent sequence
$$
\Lambda_0 < \Lambda_1 < \Lambda_2 < \ldots, \qquad 
	\Lambda_m \to \infty \quad \textrm{when} \quad m \to \infty.
$$
Moreover, we require the associated `partial derivatives' $\delta_i$ 
to be $L^2$-consistent with the orthogonal system, i.e. for each $i$
$$
\delta_i \varphi_{k_i}^{(i)} \in L^2(X,\mu_i), \qquad k_i \in \mathbb{N},
$$
and
$$
\langle \delta_i \varphi_{k_i}^{(i)}, \delta_i \varphi_{m_i}^{(i)} \rangle_{\mu_i}
= \langle \delta_i^* \delta_i \varphi_{k_i}^{(i)}, \varphi_{m_i}^{(i)} \rangle_{\mu_i}, \qquad
k_i,m_i \in \mathbb{N}.
$$
All these assumptions are not too restrictive, as may be seen by various examples given in 
\cite[Section 7]{NS1}.

In the situation described above it is not hard to check that the `Laplacian' $L$ is symmetric
and nonnegative on $C_c^2(\mathcal{X})\subset L^2(\mathcal{X},\mu)$, and the constant $A$ in the 
decomposition of $L$ does not exceed the smallest eigenvalue, $A \le \Lambda_0$,
cf. \cite[Lemma 1]{NS1}.
However, from the conjugacy point of view, the following fact is essential (see \cite[Lemma 2]{NS1}):
given $i=1,\ldots,d$, the `differentiated' system $\{\delta_i \varphi_k : k\in \mathbb{N}^d\}$
is orthogonal in $L^2(\mathcal{X},\mu)$;
furthermore, $\|\delta_i\varphi_k\|^2_{L^2(\mathcal{X},\mu)} = \lambda_{k_i}^{(i)}-a_i$.
This fact leads to investigating also `Laplacians' standing behind the systems $\{\delta_i \varphi_k\}$,
$i=1,\ldots,d$, and this turns out to be a crucial point in constructing proper conjugacy scheme for
general orthogonal expansions. Define
$$
M_j = A + \delta_j\delta_j^* + \sum_{i \neq j} \delta_i^* \delta_i = L + [\delta_j,\delta_j^*],
\qquad j=1,\ldots,d,
$$
where $[\delta_j,\delta_j^*]$ is the commutator
$$
[\delta_j,\delta_j^*] = \delta_j\delta_j^*-\delta_j^*\delta_j =
	2p_j(x_j)q'_j(x_j) - p_j(x_j)\bigg[ p_j(x_j)\frac{w'_j(x_j)}{w_j(x_j)} + p'_j(x_j)\bigg]'.
$$
By the very definition it follows that each $M_j$ is symmetric and nonnegative on 
$C_c^2(\mathcal{X})\subset L^2(\mathcal{X},\mu)$. Moreover, for each $j=1,\ldots,d$, the
system $\{\delta_j \varphi_k : k \in \mathbb{N}^d\}$ is an orthogonal system of eigenfunctions of $M_j$,
with the corresponding eigenvalues $\{\lambda_k : k \in \mathbb{N}^d\}$, i.e.
$$
M_j (\delta_j \varphi_k) = \lambda_k (\delta_j \varphi_k);
$$
see \cite[Lemma 5]{NS1}. The operators $M_j$ (or rather their suitable self-adjoint extensions)
are used to generate so-called modified Poisson semigroups that play an important role in the conjugacy
scheme for orthogonal expansions proposed in \cite{NS1}, see \cite[Section 5]{NS1} for details.

The main inconvenience of the theory postulated in \cite{NS1} is a lack of symmetry in principal
objects and relations, and this phenomenon has roots in the asymmetry between the `derivatives' $\delta_j$
and their adjoints $\delta_j^*$. 
In consequence, definitions and the conjugacy scheme established in \cite{NS1}
admit essential deviations from the classical shape, see \cite[Sections 5,6]{NS1}.

The main idea of this paper is to overcome the problem by
embedding the situation considered in \cite{NS1} into a more general setting, where
the associated derivatives are skew-symmetric and the related conjugacy scheme has precisely the classical
shape. The price is, however, that the related extended `Laplacian' and `derivatives' 
are differential-difference
operators rather than purely differential ones. It is remarkable that most definitions and relations
in the setting of \cite{NS1} may be then recovered by suitable `projecting' from the extended
situation. However, in some cases the projection procedure leads to different and seemingly even more
natural definitions. This remark concerns especially higher order Riesz transforms.

\section{Symmetrization} \label{sec:sym}

We now describe the symmetrization procedure and the resulting symmetrized situation. 
The construction is motivated to some extent by the setting of
the Dunkl harmonic oscillator with the underlying reflection group isomorphic to 
$\mathbb{Z}_2^d = \{0,1\}^d$;
we refer to \cite{R1} for more details concerning the Dunkl setting.
Recall that $X=(0,c)$ for some $0<c\le \infty$ and $\mathcal{X}=X^d$, and
consider the space $\mathbb{X}=X_{\textrm{SYM}}\times \ldots \times X_{\textrm{SYM}}$ ($d$-times), where
$X_{\textrm{SYM}} = (-c,0)\cup (0,c)$.
Notice that $\mathcal{X}$ is isomorphic to each of the
`Weil chambers' generated in $\mathbb{X}$ by reflections perpendicular to coordinate axes. 
We extend the measure $\mu$ to $\mathbb{X}$ by even extension of
the one-dimensional densities $w_i$, $w_i(-x_i) = w_i(x_i)$, $x_i>0$; we keep using the same symbols for the
extended objects. Further, we extend the coefficients of $L$ by letting
$$
p_i(-x_i) = p_i(x_i), \qquad q_i(-x_i) = -q_i(x_i), \qquad x_i >0;
$$
again the emerging extended objects, including $L$, $M_j$ and $\delta_j$
defined by means of the extended coefficients, are denoted by still the same symbols. 
\begin{defn}
For a suitable function $f$ on $\mathbb{X}$ define its `partial derivatives'
\begin{align*}
D_j f(x)  & = p_j(x_j)\frac{\partial f}{\partial x_j} (x) + q_j(x_j)\frac{f(x)+f(\sigma_j x)}{2} \\ & \quad
	+ \Big[ p_j(x_j) \frac{w'_j(x_j)}{w_j(x_j)} + p'_j(x_j)-q_j(x_j)\Big]\frac{f(x)-f(\sigma_j x)}{2},
\end{align*}
where $\sigma_j$ denotes the reflection in $\mathbb{X}$
in the hyperplane orthogonal to the $j$th coordinate axis,
$\sigma_j(x_1,\ldots,x_j,\ldots,x_d) = (x_1,\ldots,-x_j,\ldots,x_d)$.
\end{defn} 
The result below follows by integration by parts and some elementary manipulations.
\begin{propo}
The operators $D_j$, $j=1,\ldots,d$, are skew-symmetric in $L^2(\mathbb{X},\mu)$, $D_j^*=-D_j$,
in the sense that
$$
\langle D_jf,g\rangle_{\mu} = - \langle f, D_j g\rangle_{\mu}, \qquad f,g \in C^1_c(\mathbb{X}).
$$
\end{propo}
This motivates the definition of the extended `Laplacian' $\mathbb{L}$ as
\begin{equation} \label{decL}
\mathbb{L} = A - \sum_{i=1}^d D_i^2.
\end{equation}
Then each $D_i$ commutes with $\mathbb{L}$, which is an important feature at this point.

To state the next result it is convenient to introduce the following terminology. Given
$\varepsilon \in \mathbb{Z}_2^d$, we say that a function $f$ on $\mathbb{X}$ is $\varepsilon$-symmetric
if $f \circ \sigma_j = (-1)^{\varepsilon_j}f$, $j=1,\ldots,d$.
If $f$ is $\varepsilon$-symmetric and $\varepsilon_{j_0}=0$ ($\varepsilon_{j_0}=1$)
then $f$ is said to be even (odd) with respect to the $j_0$th coordinate.
\begin{propo} \label{lem:L11}
The operator $\mathbb{L}$ is symmetric and nonnegative on $C_c^2(\mathbb{X})\subset L^2(\mathbb{X},\mu)$.
Moreover, for any $\varepsilon$-symmetric function $f\in C^2(\mathbb{X})$, 
$\varepsilon \in \mathbb{Z}_2^d$, we have
$$
\mathbb{L}f = Af + \sum_{\{j : \varepsilon_j = 0\}} \delta_j^*\delta_j f + 
	\sum_{\{j : \varepsilon_j = 1\}} \delta_j \delta^*_j f.
$$
In particular, $\mathbb{L}f = Lf$ when $f$ is even with respect to all coordinates, and 
$\mathbb{L}f = M_j f$ if $f$ is odd with respect to the $j$th coordinate and even with respect to the
remaining coordinates.
\end{propo}
\begin{proof}
The first part follows from the decomposition of $\mathbb{L}$ in terms of the $D_j$. Justifying
the remaining part may be done by computing the explicit form of $\mathbb{L}$, which is
\begin{align*}
\mathbb{L}f(x) & = Af(x) - \sum_{i=1}^d \bigg\{p_i^2(x_i)\frac{\partial^2 f}{\partial x_i^2}(x) 
	+ \bigg[ 2p_i(x_i)p'_i(x_i) + p_i^2(x_i) \frac{w'_i(x_i)}{w_i(x_i)} \bigg] 
	\frac{\partial f}{\partial x_i}(x)\\
& \quad +\bigg(q_i(x_i)\bigg[p_i(x_i)\frac{w'_i(x_i)}{w_i(x_i)}+p'_i(x_i)\bigg] -q_i^2(x_i) \bigg)f(x)
	+ p_i(x_i)q'_i(x_i)\frac{f(x)+f(\sigma_i x)}{2} \\
& \quad + p_i(x_i)\bigg[p_i(x_i)\frac{w'_i(x_i)}{w_i(x_i)}+p'_i(x_i)-q_i(x_i)\bigg]'
	\frac{f(x)-f(\sigma_i x)}{2}\bigg\}.
\end{align*}
The proof is finished by comparing the above expression with the explicit forms of $\delta_i^*\delta_i$
and $\delta_i\delta_i^*$, which may be read off from the explicit expressions for $L$ and $M_j$,
see Section~\ref{sec:ini}. 
\end{proof}

Next we extend the eigenfunctions $\varphi_k$ to $\mathbb{X}$ by letting 
$\varphi_{k_i}^{(i)}(-x_i) = \varphi_{k_i}^{(i)}(x_i)$, $x_i>0$, $i=1,\ldots,d$. Then, automatically,
given $\varepsilon \in \mathbb{Z}_2^d$, the function 
$\delta_1^{\varepsilon_1}\ldots \delta_d^{\varepsilon_d}\varphi_k$ is $\varepsilon$-symmetric
by the way of extending the coefficients of $\delta_j$. It turns out that these are eigenfunctions 
of~$\mathbb{L}$.
\begin{lemma} \label{lem:L22}
Let $\varepsilon \in \mathbb{Z}_2^d$ be fixed. Then
$$
\mathbb{L} \big(\delta_1^{\varepsilon_1}\ldots \delta_d^{\varepsilon_d}\varphi_k\big)
	= \lambda_k \big(\delta_1^{\varepsilon_1}\ldots \delta_d^{\varepsilon_d}\varphi_k\big), \qquad
	k \in \mathbb{N}^d.
$$
\end{lemma}

\begin{proof}
Combine Proposition \ref{lem:L11} with the product structure of 
$\delta_1^{\varepsilon_1}\ldots \delta_d^{\varepsilon_d}\varphi_k$ and the fact that in the one-dimensional
setting $\varphi_k$ is an eigenfunction of $L=a+\delta^*\delta$ and $\delta\varphi_k$ is an eigenfunction
of $M=a+\delta\delta^*$ (to be precise, the last fact was already invoked from \cite{NS1} 
in the non-extended setting, but it easily carries over to the extended situation by the way of 
extending the coefficients of $\delta$ and $\delta^*$).
\end{proof}
Note that for a given $\varepsilon \in \mathbb{Z}_2^d$ it may happen that for some $k \in \mathbb{N}^d$
the function $\delta_1^{\varepsilon_1}\ldots \delta_d^{\varepsilon_d}\varphi_k$ vanishes identically.
This occurs precisely when there is an $i \in \{1,\ldots,d\}$ such that $\varepsilon_i=1$
and $a_i=\lambda_0^{(i)}$, and $k \in \mathbb{N}^d$ is such that $k_i=0$.

To construct an orthonormal system $\{\Phi_n\}$ associated with $\mathbb{L}$ and related to the original
system $\{\varphi_k\}$, it is natural to consider first the one-dimensional case. 
Then the relevant multi-dimensional system will be obtained simply by taking tensor products. 
The construction below is partially motivated by the case of the classical trigonometric system. Let
$$
\Phi_{n_i}^{(i)}(x_i) =
\begin{cases}
\frac{1}{\sqrt{2}} \varphi^{(i)}_{n_i\slash 2}(x_i), & n_i \;\; \textrm{even}, \\
-\frac{1}{\sqrt{2}} \big(\lambda^{(i)}_{(n_i+1)\slash 2}-a_i\big)^{-1\slash 2} 
	\, \delta_i\varphi^{(i)}_{(n_i+1)\slash 2}(x_i), & n_i \;\; \textrm{odd}.
\end{cases}
$$
In this place it seems to be natural to require that for each $i=1,\ldots,d$, the derivative
$\delta_i$ annihilates the first eigenfunction, $\delta_i \varphi_0^{(i)}\equiv 0$.
This is equivalent to assuming that the constant $a_i$ from the decomposition of $L_i$ is equal to
the first eigenvalue, $a_i = \lambda_0^{(i)}$. In the multi-dimensional setting the requirement
means the equality $A = \Lambda_0$. We emphasize that this is indeed the case, up to a convention
explained in  a moment, of all the classical examples given in \cite[Section 7]{NS1}.
Consequently, $\delta_i \varphi^{(i)}_0$ does not enter the definition of $\Phi_{n_i}^{(i)}$ above.
On the other hand, notice that all the $\Phi_{n_i}^{(i)}$ are well-defined and non-vanishing.
By the facts mentioned earlier (cf. \cite[Lemma 2]{NS1}) 
each of the systems $\{\Phi_{n_i}^{(i)} : n_i \in \mathbb{N}\}$, $i=1,\ldots,d$, is orthonormal in
$L^2(X_{\textrm{SYM}},\mu_i)$. 
For a multi-index $n=(n_1,\ldots,n_d)\in \mathbb{N}^d$ we define
$$
\Phi_n = \Phi_{n_1}^{(1)} \otimes \ldots \otimes \Phi_{n_d}^{(d)}.
$$
The multi-dimensional system $\{\Phi_n : n \in \mathbb{N}^d\}$ is orthonormal in $L^2(\mathbb{X},\mu)$.
Moreover, the $\Phi_n$ are eigenfunctions of $\mathbb{L}$, as stated below.
\begin{lemma}
We have
$$
\mathbb{L} \Phi_n = \Big( \lambda^{(1)}_{\lfloor{\frac{n_1+1}2}\rfloor}+ \ldots +
	\lambda^{(d)}_{\lfloor{\frac{n_d+1}2}\rfloor} \Big) \Phi_n, \qquad n \in \mathbb{N}^d,
$$
where $\lfloor\cdot\rfloor$ denotes the integer part function (the floor function).
\end{lemma}

\begin{proof}
Given $\varepsilon \in \mathbb{Z}_2^d$, notice that, up to a constant factor,
$\Phi_{2k-\varepsilon}$ coincides with $\delta_1^{\varepsilon_1}\ldots\delta_d^{\varepsilon_d} \varphi_k$
whenever $2k-\varepsilon \in \mathbb{N}^d$. Then Lemma \ref{lem:L22} gives the desired conclusion.
\end{proof}

In what follows, for multi-indices $n \in \mathbb{N}^d$ we will use the notation 
$$
\langle n \rangle = \bigg(\Big\lfloor \frac{n_1+1}2 \Big\rfloor , \ldots , \Big\lfloor 
	\frac{n_d+1}2 \Big\rfloor \bigg),
$$
and (in particular) $\langle n \rangle=\lfloor \frac{n+1}2 \rfloor$ when $n$ is a number.
Then we may write shortly
$$
\mathbb{L}\Phi_n = \lambda_{\langle n \rangle} \Phi_n, \qquad n \in \mathbb{N}^d.
$$

The `real' picture that emerges from the above procedure may be then turned into a `complex' one.
Indeed, define first in dimension one
$$
\Psi_{n_i}^{(i)} = \frac{1}{\sqrt{2}} \Big( \Phi^{(i)}_{2|n_i|} + \textrm{i} 
	\sgn n_i \;\Phi^{(i)}_{2|n_i|-1}\Big), \qquad n_i \in \mathbb{Z},
$$ 
and then for a multi-index $n=(n_1,\ldots,n_d)\in \mathbb{Z}^d$,
$$
\Psi_n = \Psi_{n_1}^{(1)} \otimes \ldots \otimes \Psi_{n_d}^{(d)}.
$$
An easy argument shows that the system $\{\Psi_n : n \in \mathbb{Z}^d\}$ is orthonormal in
$L^2(\mathbb{X},\mu)$ and consists of eigenfunctions of $\mathbb{L}$,
$$
\mathbb{L} \Psi_n = \lambda_{|n|} \Psi_{n}, \qquad n \in \mathbb{Z}^d,
$$
where $|n|=(|n_1|,\ldots,|n_d|)$.

We remark that the choice of signs in the construction of $\{\Phi_n\}$ is in principle arbitrary.
Our particular choice is motivated by the fundamental example below.
\vspace{5pt}

\noindent \textbf{Example 1.}
The basic example here (and in some sense a prototype) is the case of classical trigonometric expansions.
Let $d=1$ and consider the interval $\mathcal{X}=X=(0,\pi)$ equipped with the measure 
$\mu(dx)=\frac{1}{\pi}dx$. Further, consider the one-dimensional standard Laplacian 
$L=-\frac{d^2}{dx^2}$ on $(0,\pi)$ and the related orthonormal basis in $L^2(\mathcal{X},\mu)$ of cosines,
$$
\varphi_k(x) =
\begin{cases}
1, & k=0,\\
\sqrt{2} \cos kx, & k >0.
\end{cases}
$$
Clearly, $L\varphi_k = \lambda_k \varphi_k$, where $\lambda_k=k^2$. In addition $M=L=-\frac{d^2}{dx^2}$.
Applying the symmetrization procedure we arrive at the trigonometric systems
$$
\{\Phi_n : n \in \mathbb{N}\} = \Big\{ \frac{1}{\sqrt{2}}, \sin x, \cos x, \sin 2x, \cos 2x, \ldots \Big\}
$$
and
$$
\{\Psi_n : n \in \mathbb{Z}\} = \Big\{\frac{1}{\sqrt{2}} \exp(\textrm{i} nx): n \in \mathbb{Z} \Big\}
$$
on the interval $\mathbb{X}=(-\pi,\pi)$. These are orthonormal bases in $L^2((-\pi,\pi),\frac{1}{\pi}dx)$
of eigenfunctions of the Laplacian $\mathbb{L}=-\frac{d^2}{dx^2}$ considered on $(-\pi,\pi)$,
$\mathbb{L}\Phi_n =  \langle n \rangle^2 \Phi_n$, $\mathbb{L}\Psi_n =  n^2 \Psi_n$.
This example may be easily generalized to arbitrary dimension $d \ge 1$.
\vspace{5pt}

In the situation of Example 1 we could as well choose as the initial system $\{\varphi_k\}$ the system
of sines. This, however, leads to a small obstacle since then the constant in the decomposition of
$L$ does not coincide with the first eigenvalue, as required above. On the other hand, the system of
sines is commonly enumerated by $k=1,2,\ldots$, excluding $k=0$. To overcome these problems, we introduce
the following technical convention: in the case just described, and also in similar cases as those of
Fourier-Bessel systems (see \cite[Section 7.8]{NS1}), we formally treat $\lambda_0=0$ and 
$\varphi_0 \equiv 0$ as the first eigenvalue and the corresponding eigenfunction, respectively.
Then we are in a position to apply the symmetrization, which leads to an extended system $\{\Phi_n\}$
with $\Phi_0 \equiv 0$ to be neglected (thus, in fact, $\{\Phi_n\}$ is enumerated by $n=1,2,\ldots$, 
as is the initial system). 
Clearly, the convention just described in dimension one induces an analogous convention in the 
multi-dimensional situation.

Applying this convention to $\varphi_k(x) = \sqrt{2}\sin kx$, $k=1,2,\ldots$, and passing to symmetrization
we arrive at the trigonometric system $\{\Phi_n\}$ as in Example 1, but with $\Phi_0 = 1\slash \sqrt{2}$
excluded. Notice that this time the extended system is not complete. This indicates that the 
symmetrization applied to the system of cosines provides a more natural way of embedding the system
of sines into the extended symmetric situation. In other words, it is more natural to view the sines as
the `differentiated' system rather than the initial one.

\section{Riesz transforms and conjugacy} \label{sec:riesz}

In this section we investigate the symmetrized setting from the conjugacy point of view.
We define Riesz transforms and conjugate Poisson integrals associated to the extended `Laplacian'
$\mathbb{L}$, and then show that these definitions fit into a consistent conjugacy scheme.
This scheme, including Cauchy-Riemann type equations, has precisely the classical shape.
When the convention described at the end of Section \ref{sec:sym} is in force, then the results
below should be understood accordingly.

Following a general concept, we define formally the Riesz transforms 
of order $N\ge 1$ by $R^{\el}=D^{\el}\mathfrak{L}^{-|\el|\slash 2}$, $|\el|=N$; all
necessary notions will be explained momentarily. To make this definition strict, we need to specify
a suitable self-adjoint extension of $\mathbb{L}$. Consider the operator
\begin{equation} \label{str1}
{\mathfrak{L}}f = \sum_{n \in \mathbb{N}^d} \lambda_{\langle n \rangle}
	\langle f, \Phi_n\rangle_{\mu} \Phi_n
\end{equation}
defined on the domain
\begin{equation} \label{str2}
\domain \mathfrak{L} = \Big\{ f \in L^2(\mathbb{X},\mu) : \sum_{n \in \mathbb{N}^d}
	\big| \lambda_{\langle n \rangle} \langle f, \Phi_n \rangle_{\mu} \big|^2 < \infty \Big\}.
\end{equation}

We denote by $\mathcal{N} = (\spann \{\Phi_n : n \in \mathbb{N}^d, 
	\lambda_{\langle n \rangle} \neq 0\})^{\perp}$ 
the null subspace of $\mathfrak{L}$. Note that $\mathcal{N}$ is not necessarily trivial. 
Independently of the case, we shall always write $\Pi_0$ for the orthogonal projection
of $L^2(\mathbb{X},\mu)$ onto $\mathcal{N}^{\perp}$.

\begin{lemma}
The inclusion $C_c^2(\mathbb{X})\subset \domain \mathfrak{L}$ holds, so that $\mathfrak{L}$ is 
a nonnegative self-adjoint extension of the operator $\Pi_0 \mathbb{L}$ defined initially on
$C_c^2(\mathbb{X})$. Moreover, the spectrum of $\mathfrak{L}$ satisfies
$$
\{\Lambda_0,\Lambda_1,\ldots\} \subset \sigma(\mathfrak{L}) \subset 
	\{0\} \cup \{\Lambda_0,\Lambda_1,\ldots\}.
$$
\end{lemma}

\begin{proof}
Here arguments are similar to those from the proofs of \cite[Lemma 3]{NS1} and \cite[Lemma 6]{NS1}.
We omit the details.
\end{proof}

The spectral decomposition of $\mathfrak{L}$ may be written as
$$
\mathfrak{L}f = \sum_{m=0}^{\infty} \Lambda_m \mathcal{P}_m f, \qquad f \in \domain \mathfrak{L},
$$
where the spectral projections are
$$
\mathcal{P}_m f = \sum_{\{n \in \mathbb{N}^d \,:\, \lambda_{\langle n \rangle}=\Lambda_m\}}
	\langle f, \Phi_n \rangle_{\mu} \,\Phi_n, \qquad m \in \mathbb{N}.
$$

Next we define more strictly the Riesz transforms of order $N\ge 1$ by
$$
R^{\el} = D^{\el} \mathfrak{L}^{-|\el|\slash 2} \Pi_0, \qquad \el=(\el_1,\ldots,\el_d) \in
	\mathbb{N}^d \backslash \{(0,\ldots,0)\}, 
$$
where $|\el| = \el_1+\ldots+\el_d = N$ is the order of the transform, and 
$D^{\el} = D^{\el_1}_1\ldots D^{\el_d}_d$ (since the $D_j$ commute, any composition
of them may be written in such a form). Notice that for the order one, if $\el=e_j$
(the $j$th coordinate vector), then
$D^{\el}=D_j$ and consequently, $R^{\el}$ coincides with $D_j \mathfrak{L}^{-1\slash 2}\Pi_0$;
in what follows we will denote these operators by $R_j$, $j=1,\ldots,d$.
If $N \ge 2$ and $|\el|=N$, it is customary to call the operators $R^{\el}$ the Riesz transforms
of \emph{higher order}. 
To provide a fully rigorous definition of $R^{\el}$
we use the spectral series of $\mathfrak{L}$ and set
\begin{equation} \label{str3}
R^{\el}f  = \sum_{\lambda_{\langle n \rangle}\neq 0} 
	\big(\lambda_{\langle n \rangle}\big)^{-|\el|\slash 2}
	\langle f, \Phi_n \rangle_{\mu} D^{\el} \Phi_n, \qquad f \in L^2(\mathbb{X},\mu);
\end{equation}
(notice that $\lambda_{\langle n \rangle}= 0$ may happen only when $n=(0,\ldots,0)$).
To show that this formula indeed gives rise to $L^2$-bounded operators we first need to have a closer
look at the action of the `derivatives' on the eigenfunctions.
Recall that $\lambda^{(j)}_0 = a_j$, $j=1,\ldots,d$, and so $\Lambda_0=A$.

\begin{lemma} \label{lem:L44}
Given $j=1,\ldots,d$ and $N\ge 1$, we have
$$
D_j^N \Phi_n =
\begin{cases}
(-1)^{N\slash 2} \big(\lambda^{(j)}_{\langle n \rangle_j}-a_j\big)^{N\slash 2} \, \Phi_n, & 
	N \;\textrm{even},\\
(-1)^{n_j+1+(N-1)\slash 2} \big(\lambda^{(j)}_{\langle n \rangle_j}-a_j\big)^{N\slash 2} \,
	\Phi_{n-(-1)^{n_j}e_j}, & N \;\textrm{odd},
\end{cases}
$$
with the convention that $\Phi_n\equiv 0$ if $n \notin \mathbb{N}^d$.
\end{lemma}

\begin{proof}
Because of the product structure we may and do assume that $d=1$ 
(thus $k$ and $n$ are nonnegative integers). 
Recall that
$$
\Phi_{2k} = \frac{1}{\sqrt{2}} \varphi_k, \qquad 
	\Phi_{2k-1} = \frac{-1}{\sqrt{2}} \frac{1}{\sqrt{\lambda_k-a}} \delta \varphi_k.
$$
Since $\varphi_k = \sqrt{2}\Phi_{2k}$ is an even function we have
$$
D \Phi_{2k} = \delta \Big(\frac{1}{\sqrt{2}}\varphi_k\Big)
	= - \sqrt{\lambda_k-a} \Phi_{2k-1}, \qquad k \ge 0,
$$
and since $-D^2 = \mathbb{L}-a$ and $(\mathbb{L}-a)\varphi_k = (\lambda_k-a)\varphi_k$ 
(see Lemma \ref{lem:L22}),
$$
D \Phi_{2k-1} = -D^2\Big( \frac{1}{\sqrt{\lambda_k-a}} \Phi_{2k}\Big) 
	= (\mathbb{L}-a)\Big( \frac{1}{\sqrt{\lambda_k-a}} \Phi_{2k}\Big) = \sqrt{\lambda_k-a}\Phi_{2k},
	\qquad k \ge 1.
$$
Therefore,
\begin{align*}
D\Phi_n & = 
\begin{cases}
- \sqrt{\lambda_{\langle n \rangle}-a}\, \Phi_{n-1}, & n \; \textrm{even}\\
\sqrt{\lambda_{\langle n \rangle}-a} \,\Phi_{n+1}, & n \; \textrm{odd}
\end{cases}
\\
& = (-1)^{n+1} \sqrt{\lambda_{\langle n \rangle}-a} \,\Phi_{n-(-1)^n}, \qquad n \in \mathbb{N},
\end{align*}
with the convention that $\Phi_{-1} \equiv 0$.
To finish the proof it is now sufficient to observe that a double application of $D$ maps,
up to a multiplicative constant, $\Phi_n$ onto itself,
$$
D^2\Phi_n = -(\mathbb{L}-a) \Phi_n = -\big(\lambda_{\langle n \rangle}-a\big) \Phi_n.
$$
\end{proof}
Note that here, in contrast with the examples considered in \cite[Section 7]{NS1}, $D_j$ has proper
invariant subspaces that decompose orthogonally the whole subspace 
$\Pi_0 L^2(\mathbb{X},\mu)\subset L^2(\mathbb{X},\mu)$.
They are spanned by the pairs $\{\Phi_n,\Phi_{n-e_j}\}$, where $n$ is such that $n_j>0$ is even.
Notice that $D_j$ acts trivially on the subspace spanned by $\{\Phi_n : n_j=0\}$.

\begin{corollary} \label{cor:C11}
Given $\el \in \mathbb{N}^d\backslash \{(0,\ldots,0)\}$ we have
$$
D^{\el} \Phi_n = (-1)^{|\el|\slash 2 + |(n+3\slash 2)\widetilde{\el}|}
	\bigg( \prod_{j=1}^d \big(\lambda^{(j)}_{\langle n \rangle_j}-a_j\big)^{\el_j\slash 2}\bigg)
	\Phi_{n-(-1)^n\widetilde{\el}}, \qquad n \in \mathbb{N}^d,
$$
where $\widetilde{\el}$ is a multi-index such that $\widetilde{\el}_j=0$ if $\el_j$ is even and
$\widetilde{\el}_j=1$ otherwise,
$(-1)^n \widetilde{\el} = ((-1)^{n_1} \widetilde{\el}_1,\ldots,(-1)^{n_d} \widetilde{\el}_d)$ and
$(n+3\slash 2)\widetilde{\el} = 
((n_1+3\slash 2)\widetilde{\el}_1,\ldots,(n_d+3\slash 2)\widetilde{\el}_d)$.
\end{corollary}

As a consequence of the above corollary and Bessel's inequality we get the following.

\begin{propo}
The series defining the Riesz transforms $R^{\el}$ converge in
$L^2(\mathbb{X},\mu)$ and for each order $N \ge 1$ the mapping
$$
f \mapsto \bigg( \sum_{|\el|=N} |R^{\el}f|^2 \bigg)^{1\slash 2}
$$ 
is a (nonlinear) contraction in $L^2(\mathbb{X},\mu)$.
In particular, each $R^{\el}$ is a linear contraction.
\end{propo}
It is remarkable that the present approach to the higher order Riesz transforms is considerably
simpler than that in \cite[Section 4]{NS1}. This is due to the fact that in the symmetrized setting
the subspace spanned by the orthogonal system is invariant under actions of the associated `derivatives'.
More comments in this connection will be given in Section \ref{sec:ex}.

We pass to defining conjugate Poisson integrals in the symmetrized setting.
The Poisson semigroup $\{P_t\}_{t\ge 0}$ associated with $\mathfrak{L}$ is, by the spectral theorem,
given on $L^2(\mathbb{X},\mu)$ by
$$
P_t f = \exp\big(-t\mathfrak{L}^{1\slash 2}\big)f = \sum_{m=0}^{\infty}
	\exp(-t\Lambda_m^{1\slash 2}) \mathcal{P}_m f.
$$
Clearly, each $P_t$, $t\ge 0$, is a contraction on $L^2(\mathbb{X},\mu)$.
We now define the conjugate Poisson integrals $U_t^j$, $t\ge 0$, $j=1,\ldots,d$, 
as the contractions on $L^2(\mathbb{X},\mu)$ given by
$$
U_t^j f = P_t R_j f, \qquad f \in L^2(\mathbb{X},\mu).
$$
To rewrite this by means of the spectral series observe that by Lemma \ref{lem:L44}
$$
R_j f = \sum_{\lambda_{\langle n \rangle}\neq 0} (-1)^{n_i+1} \Bigg( 
	\frac{\lambda^{(j)}_{\langle n \rangle_j}-a_j}{\lambda_{\langle n \rangle}}
	\Bigg)^{1\slash 2} \langle f, \Phi_n\rangle_{\mu} \Phi_{n-(-1)^{n_j}e_j}.
$$
Since $\langle n-(-1)^{n_j}e_j\rangle = \langle n \rangle$, we see that
$$
U_t^j f = \sum_{\lambda_{\langle n \rangle}\neq 0} (-1)^{n_i+1} 
	\exp\bigg({-t\sqrt{\lambda_{\langle n \rangle}}}\bigg)
	\Bigg( 
	\frac{\lambda^{(j)}_{\langle n \rangle_j}-a_j}{\lambda_{\langle n \rangle}}
	\Bigg)^{1\slash 2} \langle f, \Phi_n\rangle_{\mu} \Phi_{n-(-1)^{n_j}e_j}.
$$
This, together with Bessel's inequality, shows that for each $t \ge 0$ also the mapping
$$
f \mapsto \sqrt{|U_t^1f|^2+\ldots + |U_t^df|^2}
$$
is a contraction in $L^2(\mathbb{X},\mu)$.

Our definitions of Riesz transforms and conjugate Poisson integrals are well motivated by the following
system of Cauchy-Riemann type equations.

\begin{propo} \label{prop:CR}
Let $f$ belong to the subspace of $L^2(\mathbb{X},\mu)$ spanned by the $\Phi_n$'s. Then
\begin{align*}
D_i U_t^j f & = D_j U_t^i f, \qquad i,j=1,\ldots,d, \\
D_j P_t f & = - \frac{\partial}{\partial t} U_t^{j} f, \qquad j=1,\ldots,d.
\end{align*}
If $A=0$, then also
$$
\sum_{j=1}^d D_j U_t^j f = \frac{\partial}{\partial t} P_t f;
$$
for $A>0$ the function $f$ on the right-hand side above must be replaced by $f-A\mathfrak{L}^{-1}\Pi_0f$.
Moreover, we have the harmonicity relations
$$
\Big( \frac{\partial^2}{\partial t^2}- \mathbb{L}\Big) P_t f = 0, \qquad
\Big( \frac{\partial^2}{\partial t^2}- \mathbb{L}\Big) U^j_t f = 0, \qquad j=1,\ldots,d.
$$
\end{propo}

\begin{proof}
It is enough to restrict the situation to $f = \Phi_n$, $n \in \mathbb{N}^d$. Since
$$
U_t^j \Phi_n = (-1)^{n_j+1} \exp\bigg({-t\sqrt{\lambda_{\langle n \rangle}}}\bigg)
	\Bigg( 
	\frac{\lambda^{(j)}_{\langle n \rangle_j}-a_j}{\lambda_{\langle n \rangle}}
	\Bigg)^{1\slash 2} \Phi_{n-(-1)^{n_j}e_j}
$$
and for $i \neq j$, $D_i \Phi_{n-(-1)^{n_j}e_j} = (-1)^{n_i+1}
	(\lambda^{(i)}_{\langle n \rangle_i}-a_i)^{1\slash 2}
	\Phi_{n-(-1)^{n_i}e_i-(-1)^{n_j}e_j}$,
the first identity follows. The second identity may be also easily justified because
$$
D_j P_t \Phi_n = (-1)^{n_j+1} \exp\bigg({-t\sqrt{\lambda_{\langle n \rangle}}}\bigg)
	\Big(\lambda^{(j)}_{\langle n \rangle_j}-a_j\Big)^{1\slash 2}
	\Phi_{n-(-1)^{n_j}e_j}
	= - \frac{\partial}{\partial t} U_t^j \Phi_n.
$$
To verify the third identity, we observe that
$$
\frac{\partial}{\partial t} P_t \Phi_n = - \sqrt{\lambda_{\langle n \rangle}}
	\exp\bigg({-t\sqrt{\lambda_{\langle n \rangle}}}\bigg) \Phi_n
$$
and since $D_j$ commutes with $\mathbb{L}$, thus also with $U_t^j$, we have
\begin{align*}
D_j U_t^j \Phi_n & = U_t^j (D_j \Phi_n) = 
	(-1)^{n_j+1} \sqrt{\lambda_{\langle n \rangle_j}-a_j} \, U_t^j \Phi_{n-(-1)^{n_j}e_j} \\
& = -\Big( \lambda^{(j)}_{\langle n \rangle_j}-a_j\Big)	
	\exp\bigg({-t\sqrt{\lambda_{\langle n \rangle}}}\bigg) 
		\big( \lambda_{\langle n \rangle}\big)^{-1\slash 2} \Phi_n.
\end{align*}
Here the last equality is obtained by recalling that 
$\langle n-(-1)^{n_j}e_j \rangle = \langle n \rangle$ and also noticing that
$n-(-1)^{n_j}e_j - (-1)^{n_j-(-1)^{n_j}}e_j = n$.
Finally, checking the harmonicity relations does not cause any problems.
\end{proof}
We remark that
a suitable information on the growth of the eigenvalues $\lambda_{\langle n \rangle}$ and on the
growth of the eigenfunctions $\Phi_n$ and their derivatives allows to show that the identities
of Proposition \ref{prop:CR} hold in fact for all $f\in L^2(\mathbb{X},\mu)$; see \cite[Proposition 5]{NS1}.

Further support for the symmetrized conjugacy scheme is provided by the identity
$$
\sum_{j=1}^d R_j^2 f = - f + A \mathfrak{L}^{-1}f, \qquad f \in \Pi_0 L^2(\mathbb{X},\mu);
$$
notice that when $A=0$ the potential term above vanishes. This is an analogue of the well-known
relation $\sum_j R_j^2 = - \textrm{Id}$, satisfied by the classical Riesz transforms 
$R_j = \partial_j (-\Delta)^{-1\slash 2}$.

A comment concerning the `complex' picture from Section \ref{sec:sym} is in order.
Note that replacing the symbols $\mathbb{N}^d$, $\langle n \rangle$ and $\Phi_n$ in \eqref{str1}
and \eqref{str2} by $\mathbb{Z}^d$, $|n|$ and $\Psi_n$, respectively, changes neither $\domain \mathfrak{L}$
nor $\mathfrak{L}$. Further, replacing $\lambda_{\langle n \rangle}$ and $\Phi_n$ in \eqref{str3}
by $\lambda_{|n|}$ and $\Psi_n$, respectively, does not change the Riesz operators
(in the one-dimensional setting, the action of $D$ on $\Psi_n$ is 
$D\Psi_n = \textrm{i} \sgn n \sqrt{\lambda_{|n|}-a}\,\Psi_n$, $n \in \mathbb{Z}$, and similarly for $D_j^N$).
Consequently, the Poisson semigroup and the conjugate Poisson integrals remain unchanged.

Finally, notice that in the context of Example 1 the Riesz transform given by \eqref{str3} for $l=d=1$
results in the classic conjugacy mapping
$$
\frac{a_0}2 + \sum_{n=1}^{\infty} \big(a_n \cos nx + b_n \sin nx \big) \mapsto
	\sum_{n=1}^{\infty} \big(b_n \cos nx - a_n \sin nx \big)
$$
($\sum_{n \in \mathbb{Z}} a_n e^{\textrm{i} nx}\mapsto 
	\sum_{n \in \mathbb{Z}} \textrm{i} \sgn n \, a_n e^{\textrm{i} nx}$ in the `complex' picture).

\section{Comments and examples} \label{sec:ex}

First we observe that the setting considered in \cite{NS1} is naturally embedded in the symmetrized
situation. Indeed, given a function $f$ on $\mathcal{X}$, consider its extension $\widetilde{f}$
to $\mathbb{X}$ that is even with respect to all coordinates. Then the definitions and relations
from the symmetrized scheme can be applied to $\widetilde{f}$, and this clearly induces analogous
restricted definitions and relations related to the original space $\mathcal{X}$.
In this way the general definitions of Riesz transforms of order one given in \cite[Section 3]{NS1}
and conjugate Poisson integrals given in \cite[Section 5]{NS1} are contained in the symmetrized
definitions from Section \ref{sec:riesz}. In a similar manner the Cauchy-Riemann type equations
and harmonicity relations \cite[(5.3)-(5.6)]{NS1} are `projections' of the identities from Proposition
\ref{prop:CR}. 
Moreover, by considering extensions of a function $f$ that are odd with respect to one coordinate
and even with respect to all remaining coordinates it can be seen that the `supplementary' operators
and relations established in \cite[Section 6]{NS1} are suitable `projections' of the symmetrized
counterparts from Section \ref{sec:riesz}.

However, the definition of higher order Riesz transforms induced by the symmetrized scheme
in the initial setting is essentially different from that postulated in \cite[Section 4]{NS1}.
Nevertheless, it seems to be far more appropriate and natural. Observe, that the `projection' from
the symmetrized situation via considering functions that are even with respect to all coordinates
leads to higher order derivatives in the initial setting that are of the form
$$
\mathbb{D}^n = 
	\big(\underbrace{\ldots \delta_1 \delta_1^* \delta_1 \delta_1^* \delta_1}_{n_1\; \textrm{components}} \big)
	\big(\underbrace{\ldots \delta_2 \delta_2^* \delta_2 \delta_2^* \delta_2}_{n_2\; \textrm{components}} \big)
	\ldots
	\big(\underbrace{\ldots \delta_d \delta_d^* \delta_d \delta_d^* \delta_d}_{n_d\; 
	\textrm{components}} \big),
$$
where $n=(n_1,\ldots,n_d)$ is a multi-index. This obviously makes a contrast 
(when $n_j>1$ for some $j=1,\ldots,d$) with the derivatives
$$
\delta^n = \delta_1^{n_1} \ldots \delta_d^{n_d}
$$
used in \cite{NS1} to define higher order Riesz transforms.

The definition of higher order Riesz transforms in the initial setting based on $\mathbb{D}^n$
(i.e. the `even projection' of the symmetrized definition) seems to be more natural, in particular
no complications occur in connection with showing $L^2$-boundedness of these operators, see
\cite[Section 4, Section 7.9]{NS1}. The new light on understanding higher order derivatives and Riesz
transforms in the initial setting should also have an important impact on developing  
the theory of Sobolev spaces related to orthogonal expansions. This subject remains to be investigated.

We conclude the paper with several concrete examples involving selected classical orthogonal expansions, 
where the symmetrization procedure can be easily traced explicitly.
More exemplifications can be derived from those given in \cite[Section 7]{NS1}; 
in particular, we follow the notation from there.
For the sake of clarity and simplicity, in Examples 2--4 below we assume that $d=1$.
\vspace{5pt}

\noindent \textbf{Example 2.}
Let $\{h_n : n \in \mathbb{N}\}$ 
be the classical Hermite functions on $\mathbb{R}$ and consider the system $\varphi_k = \sqrt{2}h_{2k}$
on the half-line $\mathcal{X}=(0,\infty)$, $k \in \mathbb{N}$. 
This system is an orthonormal basis in $L^2(\mathcal{X},dx)$ consisting of eigenfunctions
of the harmonic oscillator $L= -\frac{d^2}{dx^2}+x^2$ restricted to $(0,\infty)$.
The related derivatives decomposing $L$ are $\delta = \frac{d}{dx}+x$ and $\delta^*=-\frac{d}{dx}+x$,
see \cite[Section 7.4]{NS1} (notice that $\delta$ is not skew-symmetric). Passing to the
symmetrized situation we get the orthonormal system $\{\Phi_n\}$ in $L^2(\mathbb{R},dx)$, which
coincides, up to signs, with the full system of Hermite functions. The symmetrized derivative is
$
Df = \frac{df}{dx}+x \check{f},
$
where $\check{f}(x)=f(-x)$ is the reflection of $f$, and the symmetrized `Laplacian' has the form
$$
\mathbb{L}f = -\frac{d^2 f}{dx^2} + x^2f + 2f_{\textrm{odd}},
$$
with $f_{\textrm{odd}}=(f-\check{f})\slash 2$
being the odd part of $f$. Notice that $\mathbb{L}$ differs from the harmonic
oscillator by the reflection term above. On the other hand, the derivative $D$ is formally skew-adjoint
in $L^2(\mathbb{R},dx)$.
\vspace{5pt}

\noindent \textbf{Example 3.}
A natural generalization of the previous example is obtained by taking $\mathcal{X}=(0,\infty)$
equipped with the measure $\mu_{\alpha}(dx)=x^{2\alpha+1}dx$, $\alpha > -1$, and considering
the system $\varphi_k = \ell_k^{\alpha}$ of Laguerre functions of convolution type, see
\cite[Section 7.6]{NS1}. Here $\alpha$ is a parameter of type, and the value $\alpha=-1\slash 2$
corresponds to the situation described in Example 2. The related standard `Laplacian' is
$$
L = -\frac{d^2}{dx^2} - \frac{2\alpha+1}{x} \frac{d}{dx} + x^2
$$
and the associated derivatives are of the form $\delta=\frac{d}{dx}+x$, 
$\delta^* = -\frac{d}{dx}+x-\frac{2\alpha+1}{x}$.
Passing to the symmetrized situation we arrive at the system $\{\Phi_n\}$ that coincides, up to signs,
with the system of generalized Hermite functions emerging in the context of the Dunkl harmonic oscillator
and the underlying group of reflections isomorphic to $\mathbb{Z}_2$, see \cite{R1}.
However, the symmetrized `Laplacian'
$$
\mathbb{L}f = -\frac{d^2 f}{dx^2} - \frac{2\alpha+1}{x} \frac{df}{dx} + x^2f 
	+ \frac{2\alpha+1}{x^2}f_{\textrm{odd}} + 2f_{\textrm{odd}}
$$
differs from the Dunkl harmonic oscillator by the term $2f_{\textrm{odd}}$ above. The symmetrized
derivative $Df = \frac{df}{dx}+x\check{f} + \frac{2\alpha+1}{x}f_{\textrm{odd}}$ is
skew-symmetric, which is not the case of $\delta$.
\vspace{5pt}

\noindent \textbf{Example 4.}
Finally, consider an orthonormal basis $\{\varphi_k\}$ of $L^2(\mathcal{X},\mu)$ consisting of
eigenfunctions of a divergence form operator
$$
Lf = - \frac{1}{w} \big( wf'\big)' + af = - \frac{d^2f}{dx^2} - \frac{w'}{w} \frac{df}{dx} + af,
$$
where $w$ is the density of $\mu$ and $a \ge 0$ is a constant. We assume that all the technical assumptions
from Section \ref{sec:ini} are satisfied, in particular $\mathcal{X}$ is an interval of the form $(0,c)$,
$0 < c \le \infty$. The derivatives decomposing $L$ have the form $\delta=\frac{d}{dx}$,
$\delta^* = -\frac{d}{dx}-\frac{w'}{w}$. Performing the symmetrization procedure, we find the symmetrized
`Laplacian'
$$
\mathbb{L}f = -\frac{d^2f}{dx^2} - \frac{w'}{w}\frac{df}{dx} + af - \Big(\frac{w'}{w}\Big)'f_{\textrm{odd}}
$$
and the associated derivative
$$
Df = \frac{df}{dx} + \frac{w'}{w} f_{\textrm{odd}},
$$
which is skew-symmetric in $L^2(\mathbb{X},\mu)$.

A special case of the situation just described occurs when $\varphi_k$ are the (normalized) Hermite
polynomials of successive even orders, $\mathcal{X}=(0,\infty)$, $w(x)=e^{-x^2}$, $a=0$, and
$L = -\frac{d^2}{dx^2}+2x\frac{d}{dx}$ is the classical Ornstein-Uhlenbeck operator restricted to the
positive half-line. Passing to the symmetrized situation one receives the (normalized) system of
Hermite polynomials of all successive orders and the symmetrized `Laplacian' $\mathbb{L}$ which differs
from the Ornstein-Uhlenbeck operator by the reflection term $2f_{\textrm{odd}}$. The point, however, is
that the associated derivative is skew-symmetric.

Another important special case is obtained by choosing $\varphi_k$ to be the normalized Jacobi
trigonometric polynomials considered on the interval $\mathcal{X}=(0,\pi)$ equipped with the measure
$d\mu(\theta)=w(\theta)d\theta = (\sin\frac{\theta}2)^{2\alpha+1}(\cos\frac{\theta}2)^{2\beta+1}d\theta$.
Here $\alpha,\beta>-1$ are parameters of type, and taking $\alpha=\beta=-1\slash 2$ we recover the
situation of cosine expansions already discussed in Example 1. The related `Laplacian' is
$$
L = -\frac{d^2\theta}{d\theta^2} - \frac{\alpha-\beta+(\alpha+\beta+1)\cos\theta}{\sin\theta}
	\frac{d}{d\theta} + \Big(\frac{\alpha+\beta+1}2\Big)^2
$$
and the derivatives decomposing it have the form $\delta = \frac{d}{d\theta}$,
$\delta^* = -\frac{d}{d\theta} - (\alpha+\frac{1}2)\cot\frac{\theta}2+(\beta+\frac{1}2\tan\frac{\theta}2)$.
The symmetrized `Laplacian' is then
\begin{align*}
\mathbb{L}f & = -\frac{d^2f}{d\theta^2} - \frac{\alpha-\beta+(\alpha+\beta+1)\cos\theta}{\sin\theta}
	\frac{df}{d\theta} + \Big(\frac{\alpha+\beta+1}2\Big)^2 f\\ & \quad + 
	\frac{(\alpha+\beta+1)+ (\alpha-\beta)\cos\theta}{\sin^2\theta} f_{\textrm{odd}}
\end{align*}
and the associated skew-symmetric derivative is
$$
Df = \frac{df}{d\theta} + \frac{\alpha-\beta + (\alpha+\beta+1)\cos\theta}{\sin\theta} f_{\textrm{odd}}.
$$
It is worth to note that this $D$ coincides with the Jacobi-Dunkl operator on the interval $(-\pi,\pi)$,
and the extended system in the `complex' picture $\{\Psi_n\}$ consists of trigonometric polynomials
called the Jacobi-Dunkl polynomials; see \cite{frej}, for instance.

\vspace{5pt}

Developing widely understood harmonic analysis for orthogonal expansions is intimately connected
with (sometimes implicit) choice of the associated `Laplacian'. The results of this paper, and
in particular the examples given above, show that in many cases there are reasonable and in some
aspects more natural alternatives for standard `Laplacians' related to various orthogonal systems
appearing in the literature. From this point of view deleting the constant $A$ in the decomposition
\eqref{decL} of $\mathbb{L}$ would lead, in some sense, to canonical `Laplacian' associated to
general orthogonal expansions, bringing the related harmonic analysis closer to the classical case.

\end{document}